\documentclass[a4paper,reqno,oneside]{article} 
\include{mathcommands.extratex}

\begin{document}


\title{A classification of overlapping clustering schemes for hypergraphs}
\author{Vilhelm Agdur}
\maketitle
\begin{abstract}
    Community detection in graphs is a problem that is likely to be relevant whenever network data appears, and consequently the problem has received much attention with many different methods and algorithms applied. However, many of these methods are hard to study theoretically, and they optimise for somewhat different goals. A general and rigorous account of the problem and possible methods remains elusive.

    We study the problem of finding overlapping clusterings of hypergraphs, continuing the line of research started by Carlsson and Mémoli~\cite{carlsson_mémoli_2013} of classifying clustering schemes as functors. We extend their notion of representability to the overlapping case, showing that any representable overlapping clustering scheme is excisive and functorial, and any excisive and functorial clustering scheme is isomorphic to a representable clustering scheme.

    We also note that, for simple graphs, any representable clustering scheme is computable in polynomial time on graphs of bounded expansion, with an exponent determined by the maximum independence number of a graph in the representing set. This result also applies to non-overlapping representable clustering schemes, and so may be of independent interest.
\end{abstract}

\section{Introduction}

The problem of inferring community structure from a graph is one that appears in many different guises in different research areas. It is useful for everything from assigning research papers to subfields of physics~\cite{community_structure_physrev}, assessing the development of nano-medicine by studying patent cocitation networks~\cite{canadian_nanotech_patents}, studying brain function via the community structure of the connectome~\cite{connectome_analysis}, understanding how different bird species spread various plant seeds~\cite{bird_seed_networks}, to studying hate speech on Twitter~\cite{twitter_hate_speech}.

This problem is related to a broader problem of clustering, where the data may take the form of dissimilarity metric on items. In this guise, the study of clustering has a long history, with an early example being the book by Jardine and Sibson~\cite{Jardine_Sibson_1971} on mathematical taxonomy from 1971. In fact, reading the introduction to that book already prefigures many of the issues discussed later in the literature, such as the difficulty of determining what exactly is meant by ``community structure'' or ``clustering'', and the lack of distinction between algorithms for finding partitions and the definitions of the structures we attempt to find.

Unfortunately, the development since then has seen the development of many directions of research and varying methods, without much contact between the different approaches. This has lead to a confusion of various methods for community detection on graphs, where there is rarely any clear theoretical justification for preferring one method over another.

In the literature on community detection on graphs, the most common approach is to define a quality function on partitions, and then treat community detection as an optimization problem.

One of the most popular methods in practice, namely modularity optimization~\cite{NewmanGirvan}, is of this type -- it is the method used in all the examples given in the first paragraph. Other noteworthy methods in broadly the same strand of research are the flow-based infomap method motivated by information theory~\cite{infomap}, more statistically sound methods assuming a generative model~\cite{Peixoto_clustering_method}, novel graph neural network methods~\cite{node2vec_clustering}, but many more exist~\cite{multidisciplinary_review_methods}.

What the methods in this research direction have in common is that they are generally black boxes  -- they do not give you any interpretable way of answering the question of \emph{why} two vertices were put in the same part. They are also generally hard to study in a rigorous mathematical way, resulting in few theoretical guarantees for their behaviour, and computationally, the optimization problem is normally not tractable.

For example, modularity suffers from a resolution limit, wherein it cannot ``see'' things happening in subgraphs that are small compared to the total graph, resulting in very unexpected behaviour~\cite{resolution_limit}. In Figure~\ref{fig:resolution_limit} we see one illustration of this phenomenon -- one would hope that what happens in one connected component cannot affect the other, but the resolution limit results in this undesirable behaviour.

\begin{figure}
    \centering
    \subfloat[\centering A graph $G$, with vertices coloured according to the modularity-maximising partition]{{\includegraphics[height=5cm]{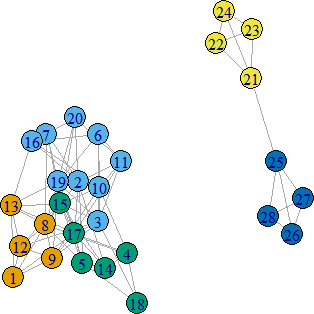} }}%
    \qquad
    \subfloat[\centering $G$ with a single edge added, highlighted in red, and its new optimal partition]{{\includegraphics[height=5cm]{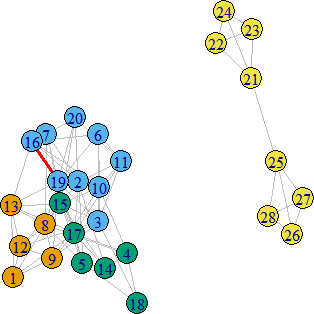} }}%
    \caption{An illustration of the resolution limit issue with modularity.}
    \label{fig:resolution_limit}
\end{figure}


In the literature on clustering based on dissimilarity measures, there has been much more attention to simpler methods, developing from the earliest studies of single-linkage clustering to more sophisticated versions of the same idea. In single-linkage clustering, we simply declare two objects to belong to the same part if they are sufficiently close, which leads to a very tractable and easily-analyzed method. However, it struggles with ``chaining'' effects, where two objects that are far apart are joined together by a chain of objects that are all close to each other.

This class of methods has received sophisticated theoretical analysis beginning with the work of Carlsson and Mémoli~\cite{carlsson_mémoli_2013}, casting the problem of determining possible clustering methods in category-theoretical language as understanding properties of functors from the category of finite metric spaces to the category of sets with partitions. This line of research has also been extended into settings that cannot be represented as finite metric spaces, such as directed graphs~\cite{Carlsson_Memoli_Segarra_2021}, and the original work of Carlsson and Mémoli straightforwardly extends into the hypergraph setting without any major modifications.

In this line of research, we end up with three related concepts to describe our clustering methods:
\begin{enumerate}
    \item \emph{Functoriality} is essentially a form of monotonicity under addition of edges or vertices to the graph. Adding a new edge or new vertex should only ever make communities more closely connected, it should never split them.
    \item \emph{Excisiveness} is a very strict version of requiring that there be no resolution limit -- instead we require that the method be idempotent on the parts it has found. That is, if we take some graph $G$, cluster it, and then look at the subgraph induced by any part, that graph has to be clustered into just a single large component, so we never discover new features at smaller scales.
    \item \emph{Representability} is a type of explainability of the clustering method. A representable clustering method $\Phi_\Rset$ is determined by a set $\Rset$ of graphs, and the rule that two vertices of a graph $G$ will be in the same part if their neighbourhood looks like something in $\Rset$. This also gives us, for simple graphs, an algorithm to actually compute the partitioning of a graph that runs in polynomial time for each fixed $\Rset$.
\end{enumerate}

Sometimes additional concepts are studied, but the basic picture is that we expect a clustering scheme to be representable if and only if it is functorial and excisive -- this is the well-known picture for simple graphs.

We extend these results into the setting of finding \emph{overlapping} partitions of hypergraphs, where vertices are allowed to belong to more than one part. This direction has already begun to be explored in the setting of finite metric spaces~\cite{Culbertson_Guralnik_Hansen_Stiller_2016,Culbertson_Guralnik_Stiller_2018,Shiebler_2020}, but their results do not immediately apply to hypergraphs, since a finite metric space cannot model the full combinatorial structure of a hypergraph. They also do not extend the notion of representability to the overlapping setting.

A similar idea to ours, of finding overlapping communities by looking for certain subgraphs (``motifs'' in their language), has recently seen some attention in the applied literature as well~\cite{Chen_Tsourakakis_2024,Hajibabaei_Seydi_Koochari_2024}.

We very nearly get the same classification result as for simple graphs, except that one direction of the equivalence weakens to become only an isomorphism.

Additionally, we also note some results on the algorithmic problem of computing a representable clustering scheme, finding that some well-known results show that for graph classes of bounded expansion, we can compute the clustering in polynomial time, with exponent equal to the maximum independence number of the graphs in the representing set.

\begin{theorem}[Main results]
    Any representable clustering scheme is excisive and functorial. Any excisive and functorial clustering scheme is isomorphic to a representable clustering scheme. 
    
    Further, for any class $\mathcal{C}$ of simple graphs of bounded expansion and any finitely representable clustering scheme $\Pi_{\Rset,k}$, there exists an algorithm that computes $\Pi_{\Rset,k}(G)$ in time $O\left(\abs{V(G)}^{\alpha(\Rset)}\right)$ for any $G \in \mathcal{C}$, where $\alpha(\Rset)$ is the maximum independence number of any graph in the representing set $\Rset$.
\end{theorem}
 
Finally, we also give a structural result for when two representable clustering schemes are equal, and use this to prove that there exist representable clustering schemes that are not finitely representable, even in the setting of simple graphs only.

\section{Preliminaries}

The reader who is less familiar with category theory is invited to review the very few concepts from it that we use in Appendix~\ref{cat_theory_appendix}.

\paragraph*{List of notation}
\begin{enumerate}
    \item We use capital Roman letters for graphs, e.g. $G$ or $H$, and their vertex and edge sets are also capital Roman letters, e.g. $V, W$ and $E, F$. Likewise, when considering a partitioned set, both its underlying set and its set of parts are capital Roman letters, e.g. $V$ and $P$. 
    \item We use lowercase Roman letters for vertices or edges of graphs, e.g. $v$ or $e$, for parts of a partitioned set, e.g. $p$ or $q$, and for functions between sets and for morphisms, e.g. $f$ or $g$. One exception to this is that we use lowercase Greek letters for morphisms from a graph $R$ in a representing set $\Rset$, e.g. $\omega$.
    \item For the function sending an edge to its vertex set, we use the character $\ea$, pronounced ``e''.
    \item For functors, we use capital Greek letters, e.g. $\Phi$, $\Pi$, $\Lambda$.
    \item For clustering schemes of graphs, and also for representing sets for representable clustering schemes, we use Fraktur letters, e.g. $\Clust$, $\Rset$.
    \item If $f: X \to Y$ is a set function, we write $f_*$ for the function from $2^X$ to $2^Y$ that sends a subset of $X$ to its image under $f$ in $Y$.
\end{enumerate}

\section{Definitions of our terms}

\begin{definition}
    The category of hypergraphs $\G$ has as its objects hypergraphs, which we consider to be a triple $(V,E,\ea)$ of a finite set $V$ of vertices, a finite set $E$ of edges that is disjoint from $V$, and a function $\ea: E \to 2^V$ assigning each edge its vertex set. Notice that this allows us to have multiple edges with the same edge set.
    
    A morphism from $H = (V,E,\ea)$ to $G = (V', E',\ea')$ is an injective set function $f: V \to V'$ such that, for every $e \in E$, there is an $e' \in E'$ such that $f_*(\ea(e)) = \ea'(e')$. That is, a morphism from $H$ to $G$ is a way of realising $H$ as a sub-hypergraph of $G$.
\end{definition}

\begin{definition}
    The category of simple graphs $\G_s$ is the full subcategory of $\G$ whose objects are the simple graphs, that is, the graphs all of whose edges contain two vertices, and where no two edges have the same vertex set. That is, there are no loops or parallel edges.

    Whenever we use just the word ``graph'', we mean a hypergraph, since those are more common for us.
\end{definition}

\begin{definition}
    For a graph $G = (V,E,\ea)$ and a subset of its vertices $p \subseteq V$, we define the \emph{restriction} of $G$ to $p$, $G|_p$, by that its vertex set is $p$, and its edges are precisely those edges of $G$ all of whose vertices are contained in $p$. It is easy to see that the set inclusion function $i: p \hookrightarrow V$ will be a morphism from $G|_p$ to $G$.
\end{definition}

\begin{definition}
    The category $\Pset$ of sets with overlapping partitions has as its objects finite sets $V$ together with a set of parts $P \subseteq 2^V$.

    A morphism in $\Pset$ from $(V,P)$ to $(W,Q)$ is a set function $f: V \to W$ such that for every $p \in P$, there exists a $q \in Q$ such that $f_*(p) \subseteq q$.
\end{definition}

It is worth remarking here that this is a very permissive definition of a partitioned set. It does not require that the union of the parts cover the entire set, nor does it require that the parts be disjoint. In fact, it even permits the cases where $\emptyset \in P$ or $\emptyset = P$.

\begin{remark}
    Notice that this definition does not exclude the case of one part properly containing another part. However, if we have a partitioned set $(V,P)$ with two parts $p \subsetneq q$, we can always remove the part $p$ to get a new partitioned set $(V, P \setminus \{p\})$, which will be isomorphic to $(V, P)$ in $\Pset$ via the identity function on the set $V$. We will call such parts which are properly contained in other parts \emph{spurious} parts.

    So, up to isomorphism, we need not care about such parts -- however, allowing them in our definition makes some of our constructions clearer, and so we choose to keep them in our definition. One might think that we should instead have picked a stronger notion of morphism, to make these be non-isomorphic, but as we will see in Example~\ref{ex:scandalous_morphism}, this would break the functoriality of our generalized notion of connected components.
    
    In particular, partitioned sets of the form $(V, \{V,W\})$ arise as the image of graphs under this connected components functor, and morphisms between graphs turn into morphisms from $(V, \{V,W\})$ to $(V, \{V,W'\})$ for arbitrary choices of $W$ and $W'$. So the fact that the spurious part $W$ was allowed to move around arbitrarily is in fact a feature, not a bug, of our definition.
\end{remark}

\begin{definition}
    The category $\Psetno$ of sets with non-overlapping partitions is the full subcategory of $\Pset$ whose objects have disjoint set covers. We may equivalently think of it as the category of sets equipped with partial equivalence relations, where the equivalence classes are the parts of the partitioned set.
    
    The partial-ness of the equivalence relations here comes from the fact that we do not require every vertex to be contained in a part. This is just an artifact of our definitions, which lead to us saying that an isolated vertex belongs to no part, instead of belonging to a part of its own. While this is somewhat unnatural, the only ``fix'' to this would be to require a loop at every vertex of a graph, which comes with its own awkwardness.
\end{definition}

\begin{definition}
    A \emph{(general) clustering scheme} is a function $\Clust$ that sends hypergraphs to partitioned sets, with the same underlying set -- that is, if $\Clust$ sends $G = (V,E,\ea)$ to $P = (W, P)$, we require that $V = W$. If the image of $\Clust$ is in $\Pset$, we say that it is a \emph{non-overlapping} clustering scheme.
    
    If, in addition, $\Clust$ becomes a functor from $\G$ to $\Pset$ by defining $\Clust(f) = f$ as functions on the underlying set, we say that it is \emph{functorial}.
    
    We may abuse our notation slightly and write $p \in \Clust(G)$ when what we mean is that $\Clust(G) = (V,P)$ and $p \in P$.
\end{definition}

One good way to think of these definitions is as a kind of generalization of monotonicity. If we were restricting to only studying clustering of graphs on a fixed underlying set with all morphisms as the inclusion map, the statement that there is a morphism from $G$ to $H$ becomes precisely the statement that $G \subseteq H$, that is that $G$ is less than $H$ in the subgraph inclusion order. Likewise, if we restricted to only studying non-overlapping partitions, there is a natural partial order on the set of equivalence relations on the underlying vertex set, given by that $\sim < \approx$ if $\sim$ is a refinement of $\approx$, or equivalently if $\approx$ is a coarsening of $\sim$.

In this restricted setting, the statement that a clustering scheme is functorial becomes exactly the statement that it is monotone with respect to the inclusion order and the order on equivalence relations. By using the language of category theory to encode this information, we gain the ability to talk about things being ``monotone under addition of vertices'', not just under addition of edges, and it becomes easier to speak about monotonicity with respect to non-overlapping set partitions.

\begin{definition}
    A clustering scheme $\Clust$ is \emph{excisive} if, for all graphs $G = (V, E, \ea)$ and all parts $p \in \Clust(G)$, $p$ is a part of $\Clust(G|_p)$.

    So what this means is that if we ``zoom in'' on one part of a graph, the part we started with does not break apart into smaller pieces. Notice that this does not put any restrictions on what happens to any parts contained in or overlapping $p$.
\end{definition}

In the setting of non-overlapping partitions, this property is more intuitive to state: If we cluster some graph, pick out one of the parts, and cluster just the subgraph induced by that part, we get back a trivial partition into a single part. Adapting this into the overlapping clustering world unfortunately gets us a less transparent definition.

\section{Representability}

\begin{definition}
    Given a not necessarily finite set $\Rset$ of graphs, we define the endofunctor $\Phi_\Rset$ on $\G$ by that, for any graph $G = (V,E)$,
    \begin{enumerate}
        \item the vertices of $\Phi_\Rset(G)$ are simply $V$,
        \item the edges of $\Phi_\Rset(G)$ are given by
        $$E(\Phi_\Rset(G)) = \bigcup_{R \in \Rset} \hom(R, G),$$
        the set of all morphisms from a graph in $\Rset$ into $G$,
        \item and the edge-assigning function $\ea_{\Phi_\Rset(G)}: E(\Phi_\Rset(G)) \to 2^V$ simply assigns each morphism to its image, i.e. for $\omega \in \hom(R, G)$, we set $\ea_{\Phi_\Rset(G)}(\omega) = \omega_*(R)$.
    \end{enumerate}

    On morphisms, $\Phi_\Rset$ is always just the identity on set functions, so if $f: G \to H$ is a morphism, $\Phi_\Rset(f) = f$ as a set function.
\end{definition}

In other words, more concretely, we get a hyperedge in $\Phi_\Rset(G)$ for every copy of a graph in $\Rset$ inside of $G$.

That this definition does in fact always give us a functor is something that needs to be checked -- it is not immediate from the definition, even though we said in the definition that it is an endofunctor.

\begin{lemma}\label{lemma_representable_endofunctor_is_functor}
    For any representing set $\Rset$ of graphs, $\Phi_\Rset$ is indeed an endofunctor on $\G$.

    \begin{proof}
        That it always gives a graph and that it respects identity morphisms and composition of morphisms is easy to see. The only thing we actually need to check is that if $f: G \to G'$ is a morphism, then $\Phi_\Rset(f) = f: \Phi_\Rset(G) \to \Phi_\Rset(G')$ is also a morphism.

        Unwrapping this statement, what we need is that, for any two graphs $(G, E, \ea)$ and $(G', E', \ea')$ and any morphism $f: G \to G'$, whenever $e$ is an edge of $\Phi_\Rset(G)$, there is an edge $e'$ of $\Phi_\Rset(G')$ such that $f_*(\ea_{\Phi_\Rset(G)}(e)) = \ea'_{\Phi_\Rset(G')}(e')$.

        Now, this $e$ is itself a morphism $\omega: R \to G$ for some $R \in \Rset$, and so if we compose $e$ with $f$ we get a morphism $e' = f \circ e: R \to G'$, which by definition is an edge of $\Phi_\Rset(G')$. That $f_*(\ea_{\Phi_\Rset(G)}(e)) = \ea'_{\Phi_\Rset(G')}(e')$ is then, if we unwrap our definitions, merely the statement that $f_*(e_*(R)) = (f \circ e)_*(R)$, which is trivial.
    \end{proof}
\end{lemma}

What is going on in the proof of Lemma~\ref{lemma_representable_endofunctor_is_functor} is much clearer if we just think about it in terms of things being subgraphs of each other, forgetting the data of \emph{how} exactly they are subgraphs. Then the statement that $\Phi_\Rset$ is a functor boils down to the statement that if $\omega$ is a subgraph of $H$ and $H$ is a subgraph of $G$, then $\omega$ is a subgraph of $G$, and so $\Phi_\Rset(G)$ must have all the edges $\Phi_\Rset(H)$ has, since edges in $\Phi_\Rset(G)$ represent the presence of a certain subgraph in $G$.

When constructing this theory for the case of non-overlapping partitions, the crucial step is to take the connected components of the graph $F_\Omega(G)$, and so get a functor taking $G$ to a set partition. By generalizing the notion of connected components a bit, we get many more options that will output also overlapping partitions.

\begin{definition}
    The functor $\Pi: \G_s \to \Psetno$ which sends a simple graph to its partition into connected components is called the connected components functor.
\end{definition}

Checking that this is indeed a functor is an easy exercise, and so we omit the proof -- it essentially boils down to the statement that adding an edge to a graph cannot disconnect a connected component, or equivalently that connectedness is a monotone increasing property in the edges.

In order to generalize this notion to hypergraphs, let us notice that for simple graphs, taking the line graph of the graph and looking at its connected components changes nothing -- the set of vertices incident to the edges of a connected component of the line graph is a connected component of the original graph.

For hypergraphs, we get more than one possible notion of line graph, and therefore we get more than one notion of connected components.

\begin{definition}
    For each $k \in \N \cup \{\infty\}$, we define the $k$-line graph functor $\Lk: \G \to \G_s$ as follows:
    
    For each $G = (V,E,\ea)$:
    \begin{enumerate}
        \item The vertices of $\Lk(G)$ are the vertex sets of the edges of $G$. That is, 
        $$V(\Lk(G)) = \left\{\ea(e) \given e \in E\right\}.$$
        Notice that this means that if there are multiple edges in $E$ with the same vertex set, we still only get one vertex of $\Lk(G)$ corresponding to them.
        \item there is an edge $uv \in E(\Lk(G))$ whenever $\abs{u \cap v} \geq k$, that is, if the two edges of $G$ overlap in at least $k$ vertices. So if $k = \infty$, we get a trivial case where the line graph has no edges.
    \end{enumerate}

    We define, for a morphism $f: H \to G$ in $\G$, that $\Lk(f) = f_*: 2^{V(H)} \to 2^{V(G)}$. This does at least send objects of the right type to objects of the right type, because the vertices of $\Lk(H)$ and $\Lk(G)$ are by construction sets of vertices of $H$ and $G$ respectively.

    However, since $V(\Lk(G))$ does not contain \emph{every} subset of $V(G)$, we do need to check that for a vertex $v$ of $\Lk(H)$, $f_*(v)$ is indeed a vertex of $\Lk(G)$. To see this, note that this $v$ is by construction the vertex set of some edge $e \in E(H)$, and since $f$ is a morphism, there exists some edge $e' \in E(G)$ such that $f_*(\ea_H(e)) = \ea_G(e')$ -- and this $\ea_G(e')$ is a vertex of $\Lk(H)$.
\end{definition}

That this line graph functor is in fact a functor definitely requires a proof, since the statement is far from obvious.

\begin{lemma}\label{lem:line_graph_is_functor}
    For each $k$, $\Lk$ is in fact a functor from $\G$ to $\G_s$.
\end{lemma}
\begin{proof}
    That $\Lk(G)$ is indeed a simple graph for each $G \in \G$ is clear, so suppose $f: G \to H$ is a morphism in $\G$. We need to check that $\Lk(f): \Lk(G) \to \Lk(H)$ is a morphism in $\G_s$.

    So, to check this, there are two things we need to show: That $\Lk(f)$ is injective as a set function, and that, for each each edge $uv \in E(\Lk(G))$, there is an edge $f(u)f(v)$ in $E(\Lk(H))$. 
    
    If we unwrap our definitions a bit, what we end up with is that this boils down to the following two statements, both of which are easily verified facts about the image of functions:
    \begin{enumerate}
        \item Whenever $f$ is an injective function, $f_*$ is as well.
        \item If $f$ is injective, then $\abs{a \cap b} = \abs{f_*(a) \cap f_*(b)}$.
    \end{enumerate}

    That $\Lk$ sends identity morphisms to identity morphisms and respects composition of morphisms is easily seen from the corresponding statements about images of functions.
\end{proof}

So, now we have the line graph functor, which gives us a simple graph which we can take connected components of. This nearly gets us where we want to be -- except of course the connected components of the line graph of $G$ will be a partitioning of the vertex sets of the edges of $G$, not a partitioning of the vertices of $G$. So our last step to getting a generalized connected components functor is to turn this into a partitioning of the vertices.

\begin{definition}
    For any set $X$, let
    $$\upsilon(X) = \bigcup_{x \in X} x,$$
    that is, we send $X$ to the union of all its elements. The reader may be momentarily confused that we are taking the union of elements of a set we did not construct as a set of sets -- how do we know that the things we are taking a union of are actually sets, so this is well-defined? However, she can find solace in the fact that, under the axioms of ZFC, \emph{all} objects are sets, and so this is well defined -- and perhaps one of very few instances where ``all objects are sets'' is useful.
    
    We define the part-union operation $\Upsilon: \Pset \to \Pset$ by that for each $\mathcal{A} = (V, P) \in \Pset$, we let $\Upsilon(\mathcal{A}) = \mathcal{B} = (\upsilon(V),\upsilon_*(P))$.
\end{definition}

Unfortunately, $\Upsilon$ cannot be made into a functor in any natural way, as we discuss in Lemma~\ref{lemma:Upsilon_not_functorial}. However, when we apply $\Upsilon$ to the connected components of the line graph, we do know what to do with the morphisms, since we still ``remember'' what they were on the underlying graph we took the line graph of. So despite not itself being a functor, we can still use $\Upsilon$ to get a functor from $\G$ to $\Pset$.

With all this in hand, we are finally able to define the connected components of a hypergraph.

\begin{definition}
    For each $k$, the $k$-ly connected components functor $\Pi_k: \G \to \Pset$ is given by that, for any $G = (V,E)$, the parts of $\Pi_k(G)$ are the parts of $\left(\Upsilon \circ \Pi \circ \Lk\right)(G)$, and the underlying set of $\Pi_k(G)$ is $V$. On morphisms, $\Pi_k(f)$ is simply the same underlying set function -- this is well defined, since clearly the underlying set of $\Pi_k(G)$ will be the same set as the vertex set of $G$.

    We say that a graph is \emph{$k$-ly connected} if $\Pi_k(G)$ has the entire vertex set of $G$ as a part, and call the parts of $\Pi_k(G)$ \emph{$k$-ly connected components} of $G$.
\end{definition}

The reason we do not simply say that $\Pi_k(G) = \Upsilon(\Pi(\Lk(G)))$ is that if there is an isolated vertex of $G$, this vertex will not be in any edge of $\Lk(G)$, and so when we apply $\Upsilon \circ \Pi$, it would not be in the underlying set. Thus our definition is the necessary choice to not forget about the isolated vertices entirely -- we do still end up declaring that they are in no part, but there seems to be no nice way of defining that problem away.

That this one is a functor also requires a proof, since it is defined as a composition of some things one of which is not itself a functor. However, the fact that the first two things are functors will be helpful in the proof.

\begin{lemma}\label{lem:k_conn_is_functor}
    For each $k$, $\Pi_k$ is a functor.
\end{lemma}
\begin{proof}
    As usual, once we have checked that $\Pi_k$ does send morphisms to morphisms, it is trivial that it respects identity morphisms and composition.

    So, suppose $G = (V,E)$ and $H = (W,F)$ are two graphs, and $f: V \to W$ is a morphism from $G$ to $H$. Let $E' = \left\{\ea(e) \given e \in E\right\}$ and $F'$ similarly be the set of vertex sets of edges of $H$.
    
    Now let $\mathcal{A} = (E',P) = \Pi(\Lk(G))$, $\mathcal{B} = (F',P') = \Pi(\Lk(G))$ be the connected components of the $k$-line graphs of $G$ and $H$ respectively. Further let $\tilde{\mathcal{A}} = (V,Q) = \Upsilon(\mathcal{A})$ and $\tilde{\mathcal{B}} = (W,Q') = \Upsilon(\mathcal{B})$.

    What we need to show is that $f$ is a morphism from $\Upsilon(\mathcal{A})$ to $\Upsilon(\mathcal{B})$ in $\Pset$. This concretely means that we need to show that, for each $q \in Q$, there is a $q' \in Q'$ such that $f_*(q) \subseteq q'$.

    Now, for each $q \in Q$, there is a part $p \in P$ such that $q = \bigcup_{e \in p} e$. By functoriality of $\Lk$ and $\Pi$, we already know that $f_*$ is a morphism from $\mathcal{A}$ to $\mathcal{B}$. This means that there has to exist a $p' \in P'$ such that $(f_*)_*(p) \subseteq p'$.

    Now, by definition
    $$(f_*)_*(p) = \left\{f_*(e) \given e \in p\right\},$$
    so if we let $q' = \bigcup_{e \in p'} e$, this must lie in $Q'$ by construction, and we get that
    \begin{align*}
        f_*(q) = f_*\left(\bigcup_{e \in p} e\right) &= \bigcup_{e \in p} f_*(e) = \bigcup_{e' \in (f_*)_*(p)} e' \subseteq \bigcup_{e' \in p'} e' = q',
    \end{align*}
    as desired.
\end{proof}

\begin{figure}
    \centering
    \subfloat[\centering A graph $G$]{{\includegraphics[height=5cm]{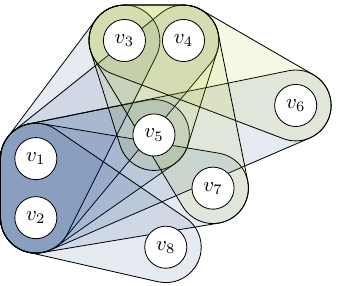} }}%
    \qquad
    \subfloat[\centering A graph $H$]{{\includegraphics[height=5cm]{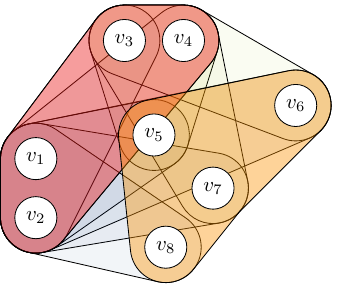} }}%
    \caption{An illustration of two graphs in $\G$, with $H$ a subgraph of $G$, where the image of this morphism under $\Lambda_2$ gives a surprising morphism in $\Pset$, as explained in Example~\ref{ex:scandalous_morphism}. Note that $H$ has all the edges of $G$, and an additional two edges, illustrated in red and orange -- the other edges have had their colour faded to help clarity, but are still present.}
    \label{fig:scandalous_morphism}
\end{figure}

\begin{remark}\label{remark:k1_case_nonoverlapping}
    In the case where $k=1$, $\Pi_1$ always gives a non-overlapping clustering, and if we also restrict ourselves to simple graphs , $\Pi_1 = \tilde{\Pi}$, where $\tilde{\Pi}$ is the usual connected components functor, except with isolated vertices assigned to no part.

    To see the first statement, suppose there were a vertex in two different parts. Then, it would be in each of these parts because it is in the vertex sets of two different edges belonging to these different parts -- but then, since $k=1$, these edges are adjacent in the $1$-line graph, and so the two parts must in fact be one part.

    For the second statement, notice that for simple graphs $\Lambda_1$ is just the usual construction of the line graph -- and so the claim is elementary graph theory.
\end{remark}

\begin{example}\label{ex:scandalous_morphism}
    At this point, we can illustrate why our class of morphisms of partitioned sets has to be so large. Consider the graphs $G$ and $H$ shown in Figure~\ref{fig:scandalous_morphism}. $G$ has eight vertices, and an edge $v_1 v_2 v_i$ for each $i > 2$, illustrated in blue, and an edge $v_3 v_4 v_j$ for $j = 5, 6, 7$, illustrated in green. $H$ is a supergraph of $G$, with all the edges of $G$ and two additional edges, $v_1 v_2 v_3 v_4$ and $v_5 v_6 v_7 v_8$, illustrated in red and orange. So that the identity function on the vertex set is a morphism from $G$ to $H$ is clear.

    Now consider the images of these two graphs under $\Pi_2$. In $G$, the blue edges will form a clique in the $2$-line graph, and the green edges will form another clique. No green edge overlaps any blue edge in more than a single vertex, so these two cliques will be the connected components, and so the parts of $\Pi_2(G)$ will be the entire vertex set and $\{v_3, v_4, v_5, v_6, v_7\}$.

    The addition of the red edge to $H$ corresponds, in the line graph, to adding a new vertex with an edge to every vertex in the two cliques, so they merge into a new connected component, whose image under $\upsilon$ will be the entire vertex set. The orange edge, however, will not be connected to any of the vertices in the two cliques, and so will form a new connected component, whose image under $\upsilon$ will be $\{v_5, v_6, v_7, v_8\}$.

    So we have seen that there must be a morphism from $(V, \{V, \{v_3, v_4, v_5, v_6, v_7\}\})$ to $(V, \{V, \{v_5, v_6, v_7, v_8\}\})$ in $\Pset$, which may at first have seemed like a bug in our definition, but is in fact a necessary feature.
\end{example}

\begin{definition}\label{definition_representable_clustscheme}
    A clustering scheme $\Clust$ is \emph{representable} if $\Clust = \Pi_k \circ \Phi_\Rset$ for some representable endofunctor $\Phi_\Rset$ and some $k \geq 1$. Notice that this means that all representable clustering schemes are functorial, being the composition of two functors.

    Given a set $\Rset$ of graphs and an integer $k$, we write $\Pi_{\Rset,k}$ for the representable clustering scheme induced by $(\Rset,k)$, that is, $\Pi_{\Rset,k} = \Pi_k \circ \Phi_\Rset$.
\end{definition}

\begin{example}
    If we restrict to considering only simple graphs with loops at every vertex, it is very easy to see that the connected components functor $\Pi$ is representable on this subcategory. In particular, we can take $\Rset = \{K_2\}$ and $k = 1$ -- then $\Phi_\Rset(G)$ is just $G$ with every edge replaced with two copies of itself, since for each edge $uv$ there are two morphisms from $K_2$ into $G$ hitting $u$ and $v$, one for each order of the two vertices.
    
    So $\Phi_\Rset(G)$ and $G$ have the same connected component, and connected components of the $1$-line graph of a simple graph and connected components of the graph itself are the same thing, and so $\Pi_{\Rset,1}(G) = \Pi(G)$.
    
    For hypergraphs, we need to take $\Rset = \left\{E_n \given n \in \N\right\}$, where $E_n$ is the hypergraph with $n$ vertices and a single edge containing all the vertices. Then all $\Phi_\Rset$ will do to a graph is to replace each edge with $k$ vertices with $k!$ copies of itself, which does not change the set of vertex sets of edges, and so again we recover the usual notion of connected components by looking at the connected components of the $1$-line graph of the graph.
    
    In fact any set $\Rset$ of connected graphs containing $E_n$ for each $n$ will represent $\Pi_1$ -- but note that they will in general give inequivalent endofunctors on $\G$.
\end{example}

It is not too hard to see that there is no \emph{finite} $\Rset$ representing $\Pi_1$ -- for if there were, there would be some $R \in \Rset$ and some $e \in E(R)$ such that $\abs{\ea(e)} = \ell$ is maximal, and so for any $j > \ell$, there would be no morphism from any $R' \in \Rset$ into $E_j$, since there are no edges of $R'$ with more than $\ell$ vertices. So $\Pi_{\Rset,1}(E_j)$ would have no parts, while clearly $\Pi_1(E_j)$ has one part containing all the vertices. So $\Pi_{\Rset,1} \neq \Pi_1$ for any finite $\Rset$.

So far, this is perhaps not too surprising -- letting the size of edges go to infinity breaks the idea of a finite representation. However, something much stronger is true: There is an infinite set $\Rset$ of \emph{simple} graphs such that $\Pi_{\Rset,1}$ is not finitely representable.

\begin{theorem} \label{theorem_exists_not_finitely_representable}
    There exists an infinite set $\Rset$ of simple graphs such that $\Pi_{\Rset,1}$ is not \emph{finitely} representable, that is, it is not equal to $\Pi_{\Gset,1}$ for any finite set of graphs $\Gset$.

    A fortiori, there exists a representable endofunctor that is not finitely representable.
\end{theorem}

In order to give a proof of this, we first need a structural result about how these endofunctors behave. Let us say that an edge in a hypergraph which contains all vertices is a \emph{spanning edge}, and a hypergraph which has a spanning edge is \emph{spanned}.

\begin{remark}\label{remark_omegas_sent_to_cliques}
    For any $R \in \Rset$, it is trivial to see that $\Phi_\Rset(R)$ must be spanned -- specifically we can take the identity morphism on $R$ as the spanning edge.
\end{remark}

In fact, which graphs are sent to spanned graphs by $\Phi_\Rset$ essentially determines what it does as a functor. Similarly, if we are interested in the weaker equivalence of inducing the same clustering scheme, this is determined by which graphs are sent to a connected graph.

\begin{lemma} \label{lemma_representation_hull}
    For any set $\Rset$ of graphs and any graph $G = (V, E, \ea)$, it holds that $\Phi_{\Rset \cup \{G\}} = \Phi_\Rset$ if and only if $G$ is sent to a spanned graph by $\Phi_\Rset$.

    \begin{proof}
        Suppose $\Phi_{\Rset \cup \{G\}} = \Phi_\Rset$. That $\Phi_{\Rset \cup \{G\}}(G)$ is spanned is immediate by Remark~\ref{remark_omegas_sent_to_cliques}, and so since $\Phi_{\Rset \cup \{G\}} = \Phi_\Rset$, $\Phi_\Rset(G)$ is spanned, as desired.

        Now suppose $\Phi_\Rset(G)$ is spanned, and let $H$ be some arbitrary simple graph. We wish to show that $\Phi_{\Rset \cup \{G\}}(H) = \Phi_\Rset(H)$. 
        
        That edges of $\Phi_\Rset(H)$ must also be edges of $\Phi_{\Rset \cup \{G\}}(H)$ is obvious, since a morphism from some graph in $\Rset$ into $H$ is also a morphism from a graph in $\Rset \cup \{G\}$ into $H$, so all we need to show is that edges of $\Phi_{\Rset \cup \{G\}}(H)$ are also edges of $\Phi_\Rset(H)$. Thus, suppose $e$ is an edge of $\Phi_{\Rset \cup \{G\}}(H)$ -- by construction, this edge is a morphism $\omega$ from some $R \in \Rset \cup \{G\}$ into $H$, whose image is the vertex set of this edge.

        If this $R$ is not $G$, then we are done. If it is $G$, we recall that $\Phi_\Rset(G)$ is spanned, and so we can take $e'$ to be a spanning edge of $\Phi_\Rset(G)$. Again, this edge is by construction a morphism $\omega'$ from some $R'$ -- this time an $R' \in \Rset$ -- into $G$, whose image, since it is spanning, must be the entirety of $G$.

        So, if we compose $\omega': R' \to G$ and $\omega: G \to H$, we get a morphism from $R'$ into $H$, which by definition is an edge of $\Phi_\Rset(H)$, and since $\omega'$ was onto, the image of this composition must be the same as the image of $\omega$, and so this edge has the right vertex set, and we are done.
    \end{proof}
\end{lemma}

\begin{lemma}\label{lemma_connected_representation_hull}
    For any set of graphs $\Rset$ and any graph $G$, if $\Pi_{\Rset,k} = \Pi_{\Rset\cup\{G\},k}$, then $\Phi_\Rset(G)$ is $k$-ly connected. The reverse implication holds only if $k = 1$.
\end{lemma}

\begin{proof}
    Suppose $\Pi_{\Rset,k} = \Pi_{\Rset\cup\{G\},k}$. We know from Remark~\ref{remark_omegas_sent_to_cliques} that $\Phi_{\Rset \cup \{G\}}(G)$ is spanned, which immediately gives that it is $k$-ly connected. So since $\Pi_{\Rset,k} = \Pi_{\Rset\cup\{G\},k}$, $\Phi_{\Rset \cup \{G\}}(G)$ and $\Phi_{\Rset}(G)$ have to have the same $k$-ly connected components, and therefore $\Phi_\Rset(G)$ is also $k$-ly connected, as desired.

    To lessen the potential confusion of terminology, we will refer to the vertices of the line graph as ``edges'' and avoid using the word edge for the edges of the line graph, only referring to paths in the line graph and to edges (i.e. line-graph-vertices) being adjacent in such paths. The word ``vertex'', similarly, refers exclusively to a vertex of $H$.

    In the other direction, for the $k = 1$ case, suppose $\Phi_\Rset(G)$ is $1$-ly connected. This means $\Upsilon(\Pi(\LOne(\Phi_{\Rset}(G))))$ has a part containing every vertex, and by definition of $\Upsilon$ and $\Pi$, this means there is some connected component $p$ of $\LOne(\Phi_{\Rset}(G))$ whose edges together contain every vertex. Let us call these edges $e_1, e_2, \ldots, e_n$, and observe that for each of these edges there exists an $R_i \in \Rset$ and an $\omega_i: R_i \to G$, where $\im(\omega_i) = \ea(e_i)$.

    So let $H$ be some arbitrary graph, for which we wish to show that $\Pi_{\Rset,1}(H) = \Pi_{\Rset\cup\{G\},1}(H)$. It is easily seen that $\Phi_\Rset(H)$ is a subgraph of $\Phi_{\Rset\cup\{G\}}(H)$, so it will suffice to show that whenever there is a path in the line graph of $\Phi_{\Rset\cup\{G\}}(H)$ between two edges $e, e'$, there is a path between them also in $\Phi_{\Rset}(H)$.

    So suppose we have such edges $e, e'$ and a path between them in $\LOne(\Phi_{\Rset\cup\{G\}}(H))$. If this path does not contain any edge induced by a copy of $G$ in $H$, then the same path exists also in $\LOne(\Phi_\Rset(H))$, and so we are done.

    So, assume $f$ is such an edge on this path, and say $\omega: G \to H$ is the morphism creating this edge. Let $e_p$ be the edge previous to $f$ on this path, and $e_s$ be the edge subsequent to $f$. Since these edges are adjacent, there has to exist $v_p \in \ea(e_p)\cap\ea(f)$ and $v_s \in \ea(f)\cap\ea(e_s)$.

    Now, since $f$ arises as a copy of $G$ in $H$, and the line graph of $G$ has a connected component covering it, whose members are the edges $e_i$, the following has to be true: There exists $i_1, i_2, \ldots, i_m$ such that $v_p \in \omega(\ea(e_{i_1}))$, $v_s \in \omega(\ea(e_{i_m}))$, and $e_{i_1}, e_{i_2}, \ldots, e_{i_m}$ form a path in the line graph of $G$.

    So, if we compose the morphisms $\omega_i$ with $\omega$, we get a collection of morphisms from graphs in $\Rset$ into $H$ -- that is, edges in $\Phi_\Rset(H)$. Call these edges $e'_1, e'_2, \ldots, e'_n$. Since $v_p \in \ea(e_p)$ and $v_p \in \ea(e'_{i_1})$, we must have that $e_p$ is adjacent to $e'_{i_1}$, and similarly $e_s$ is adjacent to $e'_{i_m}$. Notice that this is where we use the fact that $k = 1$ -- this is the step that fails for $k > 1$.

    It is easy to see that the rest of the edges on the path $e_{i_1}, e_{i_2}, \ldots, e_{i_m}$ are also still adjacent, and so we can replace $f$ on the path from $e$ to $e'$ by this longer path, getting a path that does not use $f$, and so is a path also in $\Phi_\Rset(H)$, which is what we needed to show.

    \begin{figure}
        \centering
        \subfloat[\centering The graph $G$]{
            {\includegraphics[width=0.35\textwidth,valign=c]{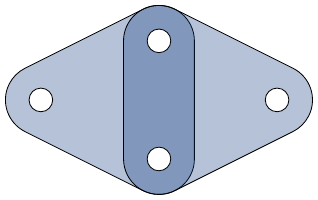}}
            \vphantom{\includegraphics[width=0.35\textwidth,valign=c]{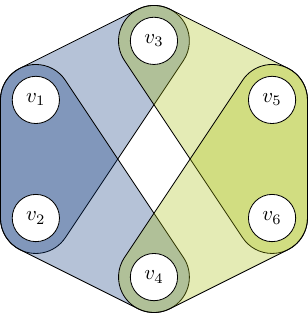}}
        }%
        \qquad
        \subfloat[\centering The graph $H$]{
            {\includegraphics[width=0.35\textwidth,valign=c]{graphics/ex_ineq_conn_hull_graphH.pdf}}
        }%
        \caption{Consider the situation where $\Rset$ consists of just $E_3$, the graph on three vertices with one edge containing all three vertices, and $G$ and $H$ are as in the figures. As we have seen previously, we will have $\Phi_\Rset(G)$ isomorphic to $G$ and likewise for $H$. It is clear that $\Phi_\Rset(G)$ is $2$-ly connected. $\Lambda_2(\Phi_\Rset(H))$ has two connected components, one containing the blue edges and one containing the green edges, and so $\Pi_{\Rset,2}(H)$ will have two parts, $\{v_1, v_2, v_3, v_4\}$ and $\{v_3, v_4, v_5, v_6\}$.
        However, when we add $G$ into $\Rset$, this will add some new edges into $\Phi_{\Rset\cup \{G\}}(H)$, with vertex sets $v_1, v_2, v_3, v_4$ and $v_3, v_4, v_5, v_6$. (There will be $4$ edges with each vertex set, due to the symmetries of $G$.) These two edge sets overlap in $v_3$ and $v_4$, and so these edges form a connected component of $\Lambda_2(\Phi_{\Rset\cup \{G\}}(H))$, which will be sent to a part containing all vertices. Thus $\Pi_{\Rset\cup\{G\},2}(H)$ is not equal to $\Pi_{\Rset,2}(H)$.}
        \label{fig:ineq_conn_hull}
    \end{figure}

    That the reverse implication fails for $k > 1$ is shown by the example in Figure~\ref{fig:ineq_conn_hull}. 
\end{proof}

With this result in hand, we can now give a proof of the existence of a representable but not finitely representable endofunctor.

\begin{proof}[Proof of Theorem~\ref{theorem_exists_not_finitely_representable}]
    Let, for each $i$, $R_i$ be the simple graph which consists of a triangle with a tail of $i$ vertices, as shown in Figure~\ref{figure_triangles_w_tails}, and let $\Rset = \left\{R_i \given i \in \N\right\}$. We claim that $\Pi_{\Rset,1}$ is not equal to $\Pi_{\Gset,1}$ for any finite $\Gset$.

    \begin{figure}
        \centering
        \includegraphics[width=0.95\textwidth]{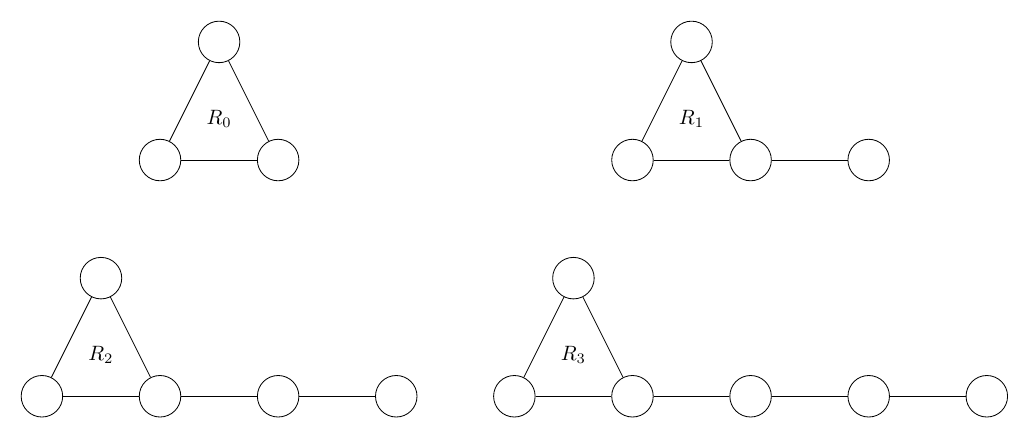}
        \caption{The family of graphs $\{R_i\}_{i=0}^3$.}
        \label{figure_triangles_w_tails}
    \end{figure}

    To prove this, let us first define
    $$\bar{\Rset} = \left\{G \in \G \given \Phi_\Rset(G) \text{ is } 1\text{-ly connected.}\right\}$$
    as the set of all graphs which $\Phi_\Rset$ sends to $1$-ly connected graphs. It follows from Lemma~\ref{lemma_connected_representation_hull} that $\Pi_{\bar{\Rset},1} = \Pi_{\Rset,1}$, and $\bar{\Rset}$ is the greatest set with this property.

    It is not hard to see that any graph in $\bar{\Rset}$ has to contain a triangle, since $\Phi_\Rset$ will send any triangle-free graph to a graph with no edges, which is obviously not connected.

    Now suppose for contradiction that $\Gset$ is a finite set of graphs such that $\Pi_{\Gset,1} = \Pi_{\bar{\Rset},1}$. Lemma~\ref{lemma_connected_representation_hull} implies that we must have $\Gset \subset \bar{\Rset}$, since if there were a $G \in \Gset \setminus \bar{\Rset}$ it would be sent to a $1$-ly connected graph by $\Phi_\Gset$, and thus also by $\Phi_{\bar{\Rset}}$, since we assumed $\Pi_{\Gset,1} = \Pi_{\bar{\Rset},1}$ -- but we chose $\bar{\Rset}$ so that it already contains everything sent to a connected graph by $\Phi_{\bar{\Rset}}$. So any graph $G \in \Gset$ contains a triangle.

    So, let
    $$r = \max_{\substack{G \in \Gset\\G\text{ simple}}} \max_{u \in V(G)} \min_{\substack{v \in V(G)\\v \text{ is in a triangle}}} d(u, v),$$
    where $d(u,v)$ is the distance between $u$ and $v$ in $G$. That is, $r$ is the furthest any vertex in any graph in $\Gset$ is from a triangle. We claim that $\Phi_\Gset$ will not send $R_{r + 1}$ to a $1$-ly connected graph, proving it is not equal to $\Phi_{\bar{\Rset}}$, giving our contradiction.

    That this is so is actually easy to see -- if $R_{r + 1}$ were sent to a $1$-ly connected graph, that must in particular mean there is an edge in $\Phi_\Gset(R_{r + 1})$ containing the tail vertex and its neighbour. So there is some morphism from some $G \in \Gset$ that hits both vertices -- and clearly there is never a morphism from a non-simple graph into a simple graph, and so this $G$ must be simple.
    
    Now, we know that $G$ contains a triangle, which has to be sent by this morphism onto the triangle in $R_{r + 1}$, since that is the only place it can go -- and since the morphism is by definition injective, this $G$ cannot contain more than one triangle, since otherwise both triangles would have to be sent onto the one triangle in $R_{r+1}$. Then, however, this morphism cannot also hit the tail of $R_{r+1}$ -- the tail is further away from the triangle than any vertex in any simple $G \in \Gset$ is from a triangle. Thus, no such morphism can exist, and we are done.
\end{proof}

\begin{remark}
    The construction in our proof of course works equally as well replacing the triangle by any connected graph that is not a line. So we have found a large family of examples of classes of graphs that give representable but not finitely representable endofunctors.

    This gives rise to a more general structural question. If we say that a graph class $\mathcal{C}$ is \emph{representable} if there exists some $\Rset$ such that $\bar{\Rset} = \mathcal{C}$, we can ask ourselves which classes of graph are representable, and which are finitely representable. Our proof above essentially relied on showing that the class of connected graphs containing a triangle is representable, but not finitely representable.

    We have shown before that the connected simple graphs are representable, with $\Rset = \{K_2\}$, and a similar construction will get the classes of connected hypergraphs of different uniformities. The example we gave in the proof easily generalises to showing that the connected graphs of girth at most $k$ are representable but not finitely representable, by extending from triangles with tails to all cycles of length at most $k$ with tails.

    The class of non-planar graphs is also representable, with itself as a representation, since there can of course be no morphism from a non-planar graph into a planar graph. Whether there exists a finite representing set is a potentially interesting question, which we leave unanswered.
\end{remark}

\section{Equivalence of representability and excisiveness}

As we saw in Remark~\ref{remark:k1_case_nonoverlapping}, we get the setting of non-overlapping partitions by setting $k=1$. In this simpler setting it is known that representability is equivalent to excisiveness plus functoriality~\cite{Carlsson_Memoli_Segarra_2021}.

In the case where $k > 1$, more interesting things start happening, so we end up not quite getting equivalence, since one of the directions of the equivalence weakens slightly. The right notion turns out to be that of one clustering scheme \emph{refining} another, which is a stronger notion than isomorphism but not quite equality.

\begin{definition}
    For any two clustering schemes $\Clust_1, \Clust_2: \G \to \Pset$ we say that $\Clust_1$ \emph{refines} $\Clust_2$ if the following two properties hold for every $G \in \G$:
    \begin{enumerate}
        \item Every part $p$ of $\Clust_1(G)$ is contained in some part $p'$ of $\Clust_2(G)$.
        \item Every part $p$ of $\Clust_2(G)$ is also a part of $\Clust_1(G)$.
    \end{enumerate}
\end{definition}

\begin{remark}
    Essentially, what this notion is saying is that for every graph $G$, $\Clust_1(G)$ is $\Clust_2(G)$, except possibly with some spurious parts added. Thus, this is a weaker notion than outright equality of clustering schemes. 
    
    However, this means that if $\Clust_1$ refines $\Clust_2$, it holds for every graph $G = (V,E)$ that the identity function on $V$ is an isomorphism between $\Clust_1(G)$ and $\Clust_2(G)$. Thus, if $\Clust_1$ refines $\Clust_2$ and both are functorial, the two clustering schemes are in fact naturally isomorphic as functors from $\G$ to $\Pset$.
\end{remark}

\begin{remark}
    If both $\Clust_1$ and $\Clust_2$ send graphs to non-overlapping partitions, it is easily seen that $\Clust_1$ refines $\Clust_2$ if and only if $\Clust_1 = \Clust_2$, so it should not be too surprising that we get refinement of clustering schemes instead of equality when we generalize from the non-overlapping case.
\end{remark}

\begin{theorem} \label{main_theorem}
    Any representable clustering scheme is excisive and functorial. Any excisive and functorial clustering scheme is refined by a representable clustering scheme, and so in particular is isomorphic to a representable clustering scheme.
\end{theorem}

We divide the proof of Theorem~\ref{main_theorem} into a collection of lemmas and examples, with the lemmas proving the stated implications and the examples demonstrating the existence of clustering schemes with every combination of properties that our lemmas do not forbid. We have already seen from Lemmas~\ref{lemma_representable_endofunctor_is_functor} and~\ref{lem:k_conn_is_functor} that representability implies functoriality, so it remains to see that it also implies excisiveness, and then to show that every excisive and functorial clustering scheme is refined by a representable clustering scheme.

A representable clustering scheme is by definition a composition of a representable endofunctor with the connected components functor -- it turns out that in general, composing a representable endofunctor with any excisive clustering scheme will give another excisive clustering scheme, so this allows us to break the problem of showing that representable clustering schemes are excisive into two parts.

\begin{lemma}\label{lem:endofunctor_circ_excisive_is_excisive}
    For any representable endofunctor $\Phi_\Rset: \G \to \G$ and any excisive clustering scheme $\Clust: \G \to \Pset$, $\Phi_\Rset \circ \Clust$ is also an excisive clustering scheme.
\end{lemma}
\begin{proof}
    The proof of this result hinges on the following claim:
    \begin{claim}\label{claim:phi_commutes_with_restriction}
        For any $G = (V,E)$ and any $p \subseteq V$, $\Phi_\Rset(G|_p) = \Phi_\Rset(G)|_p$.
    \end{claim}
    \begin{proof}
        Let $G = (V,E)$ be some graph, and let $p \subseteq V$ be arbitrary. We need to show that $\Phi_\Rset(G|_p) = \Phi_\Rset(G)|_p$. That they have the same vertex sets is obvious, and since $G|_p$ is a subgraph of $G$ and $\Phi_\Rset$ is a functor, we must have that $\Phi_\Rset(G|_p)$ is a subgraph of $\Phi_\Rset(G)$. So if we can show that any edge of $\Phi_\Rset(G)|p$ is also an edge of $\Phi_\Rset(G|_p)$, we will be done.

        So suppose $e$ is such an edge, and take $R \in \Rset$ and $\omega: R \to G$ to witness this edge. We must have that $\omega(R) \subseteq p$ -- but this means we can simply consider $\omega$ as a morphism into $G|_p$, and so it will also be an edge of $\Phi_\Rset(G|_p)$, as needed.
    \end{proof}

    With this claim in hand, proving our lemma is nearly trivial -- so let $G = (V,E)$ be some graph, and $p$ be some part of $\Clust(\Phi_R(G))$. We need to show that $p$ is a part of $\Clust(\Phi_R(G|_p))$. So if we let $H = \Phi_R(G)$, and observe that then $\Phi_R(G|_p) = H|_p$ by our claim and $p$ is a part of $\Clust(H)$ by assumption, the result is immediate by excisiveness of $\Clust$. 
\end{proof}

Proving that the $k$-ly connected components functor is excisive will be considerably trickier, because of the involvement of the non-functor $\Upsilon$.

\begin{prop}\label{lem:pi_k_is_excisive}
    The $k$-ly connected components functor $\Pi_k$ is excisive.
\end{prop}
\begin{proof}
    Let $G = (V,E)$ be some graph, and assume $\Pi_k(G) = (V,P)$. Let $p \in P$ be some part -- we need to show that $p$ is a part of $\Pi_k(G|_p)$.

    Let $p'$ be the edge set of $G|_p$, and let $p''$ be some connected component of $\Lk(G)$ which is sent to $p$ by $\upsilon$. (By Example~\ref{example:equal_parts_different_edgesets} this need not, surprisingly enough, be uniquely determined by $p$.) Notice that we have $p'' \subseteq p'$, but it can be a strict subset when $k > 1$, as demonstrated by the example in Figure~\ref{fig:overlapping_parts}.

    \begin{figure}
        \centering
        \includegraphics[width=0.5\textwidth]{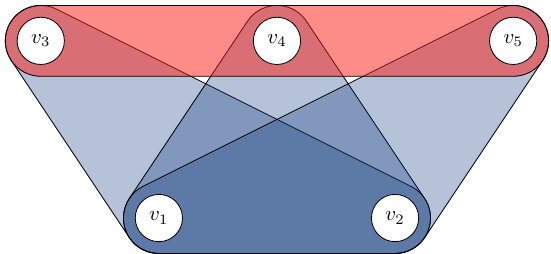}
        \caption{An example of a graph $G$ with a $2$-connected component $p$ (here, the entire vertex set) where $p''$ is a strict subset of $p'$, in the notation of the proof of Proposition~\ref{lem:pi_k_is_excisive}. In particular, we note that $p''$ consists of the three blue edges, which form a clique in the $2$-line graph since they all overlap in the vertices $v_1$ and $v_2$, but the $2$-line graph also has an isolated vertex corresponding to the red edge. This red edge is in $p'$, but not in $p''$, and will be sent to a strict subpart of $p$ by $\Upsilon$.}
        \label{fig:overlapping_parts}
    \end{figure}

    They key to this result is the following claim:
    \begin{claim}
        For any set $q'' \subseteq E$ of edges, let $q$ be the union of their vertex sets. Then it holds that $\Lk(G)|_{q''}$ is a subgraph of $\Lk(G|_q)$.\label{claim:pik_excisive_Lk_subgraph}
    \end{claim}

    Before we prove the claim, let us show why it will prove our proposition. By construction, the union of the vertex sets of the edges in $p''$ are precisely $p$, so our claim gives us that $\Lk(G)|_{p''}$ is a subgraph of $\Lk(G|_p)$. This in turn means, by functoriality of $\Pi$, that $\Pi(\Lk(G)|_{p''})$ is a sub-partitioned-set of $\Pi(\Lk(G|_p))$.

    Now it is easy to see that $\Upsilon(\Pi(\Lk(G)|_{p''}))$ and $\Upsilon(\Pi(\Lk(G|_p)))$ have the same underlying set, and since we have just seen that the parts of $\Pi(\Lk(G)|_{p''})$ are a subset of the parts of $\Pi(\Lk(G|_p))$, we get that $\Upsilon(\Pi(\Lk(G)|_{p''}))$ is a sub-partitioned-set of $\Upsilon(\Pi(\Lk(G|_p)))$, and the latter is just $\Pi_k(G|_p)$.

    So consider $\Upsilon(\Pi(\Lk(G)|_{p''}))$. By assumption $p''$ is a connected component of $\Lk(G)$, which is sent to $p$ by $\upsilon$, and so we must have that $\Upsilon(\Pi(\Lk(G)|_{p''}))$ is precisely $(p,\{p\})$. So this shows that $p$ is a part of $\Pi_k(G|_p)$, as desired.

    Finally, let us give a proof of the claim.

    \paragraph*{Proof of Claim~\ref{claim:pik_excisive_Lk_subgraph}:} Let $q'$ be the edge set of $G|_q$. That $q'' \subseteq q'$ is obvious, so the vertices of $\Lk(G)|_{q''}$ are a subset of the vertices of $\Lk(G|_q)$. So suppose $e, f \in q''$ are adjacent in $\Lk(G)$. By definition of the line graph functor, this means that $\ea(e)\cap\ea(f)$ has at least $k$ members -- but clearly $\ea(e), \ea(f) \subseteq q$, and so this intersection has the required size also in $G|_q$, and so these edges are adjacent also in $\Lk(G|_q)$.
\end{proof}

With Lemma~\ref{lem:endofunctor_circ_excisive_is_excisive} and Lemma~\ref{lem:pi_k_is_excisive} in hand, it is an easy corollary that all representable clustering schemes are excisive.

\begin{lemma} \label{lemma_repr_implies_excisive}
    Any representable clustering scheme is excisive.
\end{lemma}

\begin{lemma} \label{lemma_pi_infty_parts}
    For any graph $G$, the parts of $\Pi_\infty(G)$ are precisely the vertex sets of the edges of $G$.
\end{lemma}
\begin{proof}
    This is trivial once we notice that $\Lambda_\infty(G)$ has no edges, so its connected components are all singleton sets consisting of a single vertex, that is, of a single edge of $G$. When we then apply $\upsilon$ to these, we get the vertex sets of these edges.
\end{proof}

\begin{lemma} \label{lemma_exc_and_funct_implies_repr}
    Any excisive and functorial clustering scheme is refined by a representable clustering scheme.
\end{lemma}

\begin{proof}
    Let $\Clust$ be some excisive and functorial clustering scheme, and let
    $$\Rset = \left\{G = (V,E) \in \G \given V\text{ is a part of }\Clust(G)\right\}.$$
    We will show that $\Pi_{\Rset,\infty}$ refines $\Clust$. So, let $G$ be some graph. There are two things we need to show:
    \begin{enumerate}
        \item Every part $p$ of $\Clust(G)$ is also a part of $\Pi_{\Rset,\infty}(G)$.
        \item Every part $p$ of $\Pi_{\Rset,\infty}(G)$ is contained in some part $p'$ of $\Clust(G)$.
    \end{enumerate}

    First, suppose $p$ is a part of $\Clust(G)$. By excisiveness of $\Clust$, $\Clust(G|_p)$ has $p$ as a part, and thus $G|_p \in \Rset$. Now consider the inclusion morphism from $G|_p$ into $G$. This is a morphism from something in $\Rset$ into $G$, and so there must be an edge of $\Phi_\Rset(G)$ whose vertex set is precisely $p$. It now follows from Lemma~\ref{lemma_pi_infty_parts} that $p$ is a part of $\Pi_{\infty}(\Phi_\Rset(G)) = \Pi_{\Rset,\infty}(G)$.

    In the other direction, suppose $p$ is a part of $\Pi_{\Rset,\infty}(G)$. By Lemma~\ref{lemma_pi_infty_parts}, this must simply mean that there is an edge of $\Phi_\Rset(G)$ whose vertex set is $p$. Now, this edge arises from a morphism $\omega$ from some $R \in \Rset$ into $G$. By functoriality of $\Clust$, this morphism will also be a morphism from $\Clust(R)$ into $\Clust(G)$ whose image is $p$. By our construction of $\Rset$, $\Clust(R)$ must have a part containing all of its vertices, and so by definition of what it means to be a morphism, there has to be a part of $\Clust(G)$ containing $p$.
\end{proof}

Finally, in order to show that there is no other implication that holds between these concepts, it suffices to give examples of clustering schemes with every combination of properties not already ruled out.

\begin{example}\label{ex_allthreeprops}
    The connected components functor $\Pi_1$ is representable and excisive -- we can take its representation to be $\Rset = \left\{E_n \given n \in \N\right\}$, where $E_n$ is the graph with $n$ vertices and a single edge containing all vertices.
\end{example}
    
\begin{example}\label{ex_excnotfunctnotrepr}
    For a scheme that is excisive but not functorial (and thus not representable), consider the scheme that partitions all graphs into a single part except $K_2$, which it partitions into two parts.
\end{example}
    
\begin{example}\label{ex_functnotexcnotrepr}
    For a scheme which is functorial but neither excisive nor representable, consider the scheme $\Clust$ that is defined as follows:
    \begin{enumerate}
        \item If $G$ contains any edge containing more than two vertices, cluster all of $G$ as a single part.
        \item Otherwise, if $G$ is simple, for each connected component $p$ of $G$:
        \begin{enumerate}
            \item If $\abs{p} = 1$, cluster $p$ as a single part.
            \item If $\abs{p} = 2$, and $G = K_2$, cluster $p$ as two parts. Otherwise, cluster $p$ as a single part.
            \item If $\abs{p} > 2$, cluster $p$ as a single part.
        \end{enumerate}
    \end{enumerate}
\end{example}
    
\begin{example}\label{ex_noprops}
    Finally, for a scheme which has none of the properties, cluster all graphs as a single part except for $K_2$, which we cluster as two parts, and the disjoint union of two copies of $K_2$, which we cluster as two parts -- that is, each of the copies is its own part. To see that this fails to be functorial, and thus also fails to be representable, consider the graph with two vertices and no edges, which is sent to a single part -- but adding an edge between them splits the graph into two parts. To see that it fails to be excisive, consider the graph that is the disjoint union of two copies of $K_2$ -- it has two parts, each of which is a copy of $K_2$, but these parts will in turn be clustered as two more parts, failing excisiveness.
\end{example}

\begin{figure}
    \centering
    \includegraphics[width=0.3\textwidth]{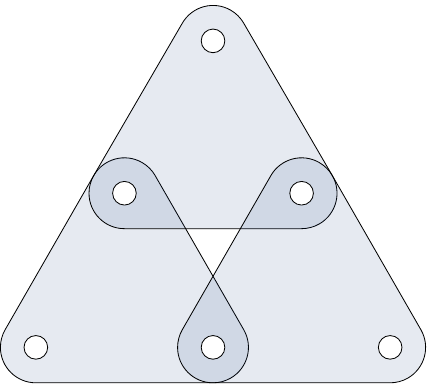}
    \caption{The graph $D$ used in the proof of Lemma~\ref{lemma:exists_exc_and_funct_not_representable}.}
    \label{fig:ex_exc_funct_not_repr-D}
\end{figure}

\begin{lemma}\label{lemma:exists_exc_and_funct_not_representable}
    There exists a clustering scheme which is excisive and functorial, but not equal to any representable clustering scheme. 
\end{lemma}
\begin{proof}
    To define this clustering scheme, we begin by letting $\Rset$ be the set containing only the graph $D$ as in Figure~\ref{fig:ex_exc_funct_not_repr-D}. Now, for any $G$, consider the graph $\Phi_\Rset(G)$. Each edge of this graph corresponds to a copy of $D$ in $G$, and so for each vertex of $\Lambda_3(\Phi_\Rset(G))$ we can give it three labels, namely the three edges of $G$ that make up the corresponding copy of $D$. So we may define a graph $\Sigma(G)$ by taking $\Lambda_3(\Phi_\Rset(G))$ and keeping an edge $ef$ if and only if $e$ and $f$ share a label.

    Concretely, this means that we will have an edge $ef$ only if they overlap in three vertices which make up an edge of both copies of $D$ to which they correspond.

    Finally, we can define our clustering scheme $\Clust$ by that $\Clust = \Upsilon\circ\Pi\circ\Sigma$, analogously to how we normally define $\Pi_{\Rset,k}$, except here $\Sigma$ plays the role of $\Lambda_3\circ\Phi_\Rset$.
    
    So there are three things we will need to show:
    \begin{enumerate}
        \item $\Clust$ is excisive,
        \item $\Clust$ is functorial, and
        \item $\Clust$ is not equal to any representable clustering scheme.
    \end{enumerate}
    For the first two, the proofs are relatively straightforward adaptations of the proof that a representable clustering scheme is excsive and of functoriality of the line graph functor, respectively, and so we will only give a sketch of the proofs here. For the final point, we end up with a complicated mess of casework, that we defer to the appendix as Lemma~\ref{lemma:excisive_and_functorial_not_representable}.

    \paragraph*{Excisiveness:} Let $G$ be some graph, and $p$ be a part of $\Clust(G)$. Recall that, from Claim~\ref{claim:phi_commutes_with_restriction}, we have that $\Phi_\Rset(G|_p) = \Phi_\Rset(G)|_p$. So let $p''$ be a set of edges of $\Phi_\Rset(G)$ that induce the part $p$. Then, entirely analogously to Claim~\ref{claim:pik_excisive_Lk_subgraph}, $\Sigma(G)|_{p''}$ is a subgraph of $\Sigma(G|_p)$. Clearly, $\Sigma(G)|_{p''}$ is a connected graph, and so $\Sigma(G|_p)$ must have a connected component containing $p''$ (in fact, $p''$ must be one of the connected components of $\Sigma(G|_p)$). This connected component is then sent to $p$ by $\upsilon$ and so $p$ is a part of $\Clust(G|_p)$.

    \paragraph*{Functoriality:} That $\Sigma$ is a functor can be proven similarly to how we prove that $\Lk$ is a functor, only keeping track of a little bit more information about the labels of the edges -- here it helps that we know $\Phi_\Rset$ is a functor. Once we know that $\Sigma$ is a functor, functoriality of $\Clust$ follows by essentially the same proof as functoriality of $\Pi_k$.
\end{proof}

\begin{remark}
    One might ask oneself why we have chosen to require all morphisms to be injective throughout. This is because in the non-overlapping case, one can relatively straightforwardly carry out the same arguments in the non-injective case~\cite{carlsson_mémoli_2013}, and find that there are then in fact only three different representable clustering schemes:
    \begin{enumerate}
        \item The scheme that always sends every vertex to its own part, represented by the empty set or by $K_1$,
        \item the connected components functor, represented by any set of connected graphs containing at least one graph on at least two vertices,
        \item and the scheme that always sends all vertices to the same part, represented by any set of graphs containing a disconnected graph.
    \end{enumerate}

    Something similar seems likely to be true in the overlapping setting, but we choose not to pursue this here, to keep the structure of the paper simpler.
\end{remark}

\section{Computational complexity of representable clustering schemes}

As stated in the introduction, one benefit of this approach to clustering is that it, for simple graphs, yields tractable algorithms for exactly computing the partitioning. In this section, we give a proof of this.

It should be clear that the only step in computing $\Pi_{\Rset,k}(G)$ that might be computationally hard is the step of computing $\Phi_\Rset(G)$, since the other steps are all obviously computable in linear time.

The most naïve possible algorithm for computing $\Phi_\Rset(G)$ given a fixed finite set $\Rset$ of representing graphs and an $n$-vertex graph $G$ is to, for each $\omega \in \Rset$ and each $H \in \binom{V(G)}{\abs{\omega}}$, that is, each $\abs{\omega}$-sized subset $H$ of the vertices of $G$, check whether $G|_H$ contains a subgraph isomorphic to $\omega$. It is easy to see that this will give an algorithm that is polynomial in $n$, but with an exponent given by the maximum size of the vertex set of any graph in $\Rset$.

What is more interesting is that we can get an algorithm that is linear in $n$ and the number of edges in $\Phi_\Rset(G)$ on any class of graphs of bounded expansion, which in turn will get us a better exponent on the $n$ from a bound on the number of edges in $\Phi_\Rset(G)$. Let us first recall what this means, before stating the lemma we will use.

\begin{definition}
    A graph $H$ is a \emph{shallow minor} of depth $t$ of a graph $G$ if it is a subgraph of a graph which can be obtained from $G$ by contracting connected subgraphs of radius at most $t$. So a shallow minor of depth $0$ is just a subgraph, while for $t$ greater than the number of vertices in $G$, a shallow minor of depth $t$ is just a minor.

    A class of simple graphs $\mathcal{C}$ is said to have \emph{bounded expansion} if for every $t \in \N$ there exists a $c_t$ such that, whenever $H$ is a shallow minor of depth $t$ of some graph in $\mathcal{C}$, it holds that
    $$\frac{\abs{E(H)}}{\abs{V(H)}} \leq c(t).$$

    This notion generalizes both proper minor closed classes and classes of bounded degree, and can equivalently be defined in terms of low tree-depth decompositions.\cite{NESETRIL_bounded_expansion}
\end{definition}

Since bounded expansion is a sparseness property, it makes sense to expect that it cannot contain too many copies of any given graph, and so in particular $\Phi_\Rset(G)$ will not have too many edges.

\begin{lemma}\label{lem:bound_edges_of_phirsetG}
    For any finite set $\Rset$ of simple graphs, let $\alpha(\Rset) = \max_{H \in \Rset} \alpha(H)$ be the maximum independence number of any graph in $\Rset$. Then, for any graph class $\mathcal{C}$, there exists a constant $C_{\Rset,\mathcal{C}}$ such that for any graph $G \in \mathcal{C}$, $\abs{E\left(\Phi_\Rset(G)\right)} = O\left(\abs{V(G)}^{\alpha(\Rset)}\right)$.
\end{lemma}
\begin{proof}
    Clearly it suffices to show this for the case $\Rset = \{H\}$ for a fixed graph $H$. Now, any graph class of bounded expansion has bounded degeneracy -- in particular, $c_0$ bounds the degeneracy.

    Now, it is known from for example~\cite[Exercise 3.2, p.~59]{sparsity_book} that if $G$ is $c_0$-degenerate and $H$ is any graph, the number of copies of $H$ in $G$ is at most
    $$\sum_{t=1}^{\alpha(H)} \mathrm{Acyc}_t(H) c_0^{\abs{H}-t} \abs{V(G)}^t,$$
    where $\mathrm{Acyc}_t(H)$ is the number of acyclic orientations of $H$ with $t$ sinks, which clearly gives us the result.
\end{proof}

To get our algorithm, we will use the following result from~\cite[Corollary 18.2, p.~407]{sparsity_book}:

\begin{lemma}\label{lemma_subgraph_counting_bounded_exp}
    Let $\mathcal{C}$ be a class with bounded expansion and let $H$ be a fixed graph. Then there exists a linear time algorithm which computes, from a pair $(G,S)$ formed by a graph $G \in \mathcal{C}$ and a subset $S$ of vertices of $G$, the number of isomorphs of $H$ in $G$ that include some vertex in $S$. There also exists an algorithm running in time $O(n) + O(k)$ listing all such isomorphs, where $k$ denotes the number of isomorphs.
\end{lemma}

Just applying this lemma for each $\omega \in \Rset$ with $S = V(G)$ immediately gets us that we can, for fixed $\Rset$, compute $\Phi_\Rset(G)$ in time linear in $n$ and the number of edges of $\Phi_\Rset(G)$, which together with Lemma~\ref{lem:bound_edges_of_phirsetG} gets us a bound on the complexity in terms of $\alpha(\Rset)$.

\begin{theorem}
    For any finite set $\Rset$ of simple graphs and any class $\mathcal{C}$ of bounded expansion, there exists an algorithm that computes $\Phi_\Rset(G)$ for any graph $G \in \mathcal{C}$, which runs in time $O\left(\abs{V(G)} + \abs{E(\Phi_\Rset(G))}\right) = O\left(\abs{V(G)}^{\alpha(\Rset)}\right)$.
\end{theorem}

\section{Open problems}
We have shown that all functorial and excisive clustering schemes are refined by a representable clustering scheme, but as Lemma~\ref{lemma:exists_exc_and_funct_not_representable} shows, this cannot be strengthened to equality. However, the construction in our example is still very similar to the construction of a representable clustering scheme, only with a more sophisticated rule for what edges we add to the line graph. It is therefore natural to ask whether there is a more general notion of a line graph, which would get us a more precise classification of clustering schemes.

We also discussed, in our proof of the existence of representable but not finitely representable clustering schemes, that the general question of what graph classes are finitely representable is interesting, and not immediately obvious.

\section*{Acknowledgments}

We thank Professors Tatyana Turova and Svante Janson and an anonymous reviewer for their helpful comments on a previous version of this paper.

\printbibliography

\appendix

\section{Definitions of categorical language}\label{cat_theory_appendix}

For the reader of a more combinatorial bent, who might not recall all the definitions of category theory used in this paper, we include a short appendix stating them.

\begin{definition}
    A \emph{category} $\mathcal{C}$ consists of a class of \emph{objects} $\mathrm{ob}(\mathcal{C})$, and for each pair of objects $a, b \in \mathcal{C}$ a (possibly empty) class $\Hom(a,b)$ of \emph{morphisms} from $a$ to $b$. 
    
    We require these morphisms to compose, in the sense that for any $f \in \Hom(a,b)$ and $g \in \Hom(b,c)$ there exists an $fg \in \Hom(a,c)$, and this composition must be associative. There must also, for each $a \in \mathcal{C}$, be an identity morphism $\mathrm{id}_a \in \Hom(a,a)$, with the property that $f\mathrm{id}_a = f$ and $\mathrm{id}_a g = g$ whenever these compositions make sense.
\end{definition}

It is worth remarking that the use of the word \emph{class} instead of \emph{set} is not accidental -- the collection of all finite simple graphs is of course too big to be a set. However, since this is only so for ``dumb reasons'' -- that there are class-many different labelings of the same graph -- and the collection of isomorphism classes of finite simple graphs \emph{is} a set, this distinction will never actually affect us, and we will elide it throughout the text.

\begin{definition}
    A \emph{functor} $F$ from a category $\mathcal{A}$ to a category $\mathcal{B}$ is a mapping that associates to each $a \in \mathcal{A}$ an $F(a) \in \mathcal{B}$, and that for each pair of objects $a, a' \in \mathcal{A}$ associates each morphism $f \in \Hom(a, a')$ to a morphism $F(f) \in \Hom(F(a), F(a'))$. We require this mapping to respect the identity morphisms, in the sense that $F(\mathrm{id}_a) = \mathrm{id}_{F(a)}$, and composition of morphisms, so that $F(fg) = F(f)F(g)$.

    A functor from a category to itself is called an \emph{endofunctor}.
\end{definition}

The categories we study in this paper are all concrete, that is, their objects are sets with some extra structure on them, and the morphisms are just functions between the sets that respect the structure in an appropriate sense. Our functors never do anything strange -- they all leave the underlying set unchanged, only changing what structure we have on the set, and they do ``nothing'' on the morphisms, that is, as set functions they are just the same function again between the same sets.

In this setting most of the requirements of a functor are trivial -- of course it will behave correctly with identity morphisms and composition, because in a sense it isn't doing anything to the morphisms. Therefore, the one thing we need to check when proving that things are functors will be that morphisms do map to morphisms, that is, that a function that respects the structure on the sets in the first category will also respect the structure we have in the second category.

The final categorical concept we will need is that of two \emph{functors} being isomorphic, since our classification result ends up giving us not equality but just isomorphism of clustering schemes.

\begin{definition}
    Suppose $\mathcal{A}$ and $\mathcal{B}$ are two categories, and $F$ and $G$ are two functors from $\mathcal{A}$ to $\mathcal{B}$. A \emph{natural transformation} $\eta$ from $F$ to $G$ is a mapping that associates to each object $a \in \mathcal{A}$ an element $\eta_a \in \Hom(F(a), G(a))$, such that for every object $b \in \mathcal{A}$ and each morphism $f \in \Hom(a, b)$, the following diagram commutes:
    $$\begin{tikzcd}
        F(a) \arrow[r, "F(f)"] \arrow[d, "\eta_a"'] & F(b) \arrow[d, "\eta_b"]\\
        G(a) \arrow[r, "G(f)"'] & G(b)
    \end{tikzcd}$$
    In other words, the following equation must hold for all $f \in \Hom(a,b)$:
    $$\eta_b \circ F(f) = G(f) \circ \eta_a.$$

    Two functors $F$ and $G$ are \emph{isomorphc} if there exists a natural transformation $\eta$ from $F$ to $G$ such that $\eta_a$ is an isomorphism for all $a \in \mathcal{A}$.
\end{definition}

This notion is very general -- in our actual result, both $\mathcal{A}$ and $\mathcal{B}$ will be concrete categories, and for each $a$, $a$, $F(a)$ and $G(a)$ will have the same underlying set, and $f$, $F(f)$, and $G(f)$ will be the same as functions on the underlying sets. So here we will be able to take each $\eta_a$ to be the identity function on the underlying set. So all the notation of the definition ends up referring to very simple things, in the end.

\section{Proof of Lemma~\ref{lemma:excisive_and_functorial_not_representable}}\label{excisive_and_functorial_not_representable_appendix}

\begin{lemma}\label{lemma:excisive_and_functorial_not_representable}
    The clustering scheme defined in Lemma~\ref{lemma:exists_exc_and_funct_not_representable} is not equal to any representable clustering scheme.
\end{lemma}

\begin{proof}
    Suppose for contradiction that $\Clust = \Pi_{\Rset,k}$ for some $\Rset$ and $k$, and assume we take $\Rset$ of minimal cardinality. We need to exclude every possible combination of $\Rset$ and value of $k$ -- so we end up with a lot of casework.
    
    \paragraph*{$\Rset$ cannot contain any proper subgraph of $D$, and must contain $D$ itself:}

    It is clear that $\Rset$ cannot contain any graph $G$ which does not have $D$ as a subgraph, since then $\Pi_{\Rset,k}(G)$ would have a part containing all vertices, while $\Clust(G)$ would have no parts at all.

    In particular this means $\Rset$ cannot contain any proper subgraph of $D$, and so, in order for $\Pi_{\Rset,k}(D)$ to have a part containing all vertices, we must have $D \in \Rset$.

    \paragraph*{Excluding the case $k \leq 3$:}

    \begin{figure}[p]
        \centering
        \includegraphics[width=0.6\textwidth]{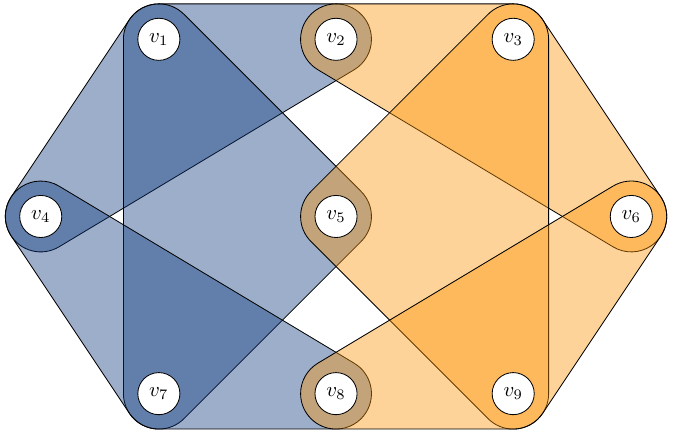}
        \caption{The graph $G$ used in the proof of Lemma~\ref{lemma:excisive_and_functorial_not_representable}.}
        \label{fig:appendix_graph_G}
    \end{figure}

    In this case, consider the graph $G$ as in Figure~\ref{fig:appendix_graph_G}. An attentive eye will see that this graph is two copies of $D$, one on the blue edges and one on the orange edges, and they overlap in their corners. It is clear that $\Clust(G)$ will have two parts, since the overlap of the two copies of $D$ is not an edge of either copy. However, since $D \in \Rset$ and $k \leq 3$, we must have a part containing all vertices in $\Pi_{\Rset,k}(G)$, which is a contradiction.

    \paragraph*{Establishing graphs that have to be in $\Rset$ if $k$ is big:}

    \begin{figure}[p]
        \centering
        \includegraphics[width=\textwidth]{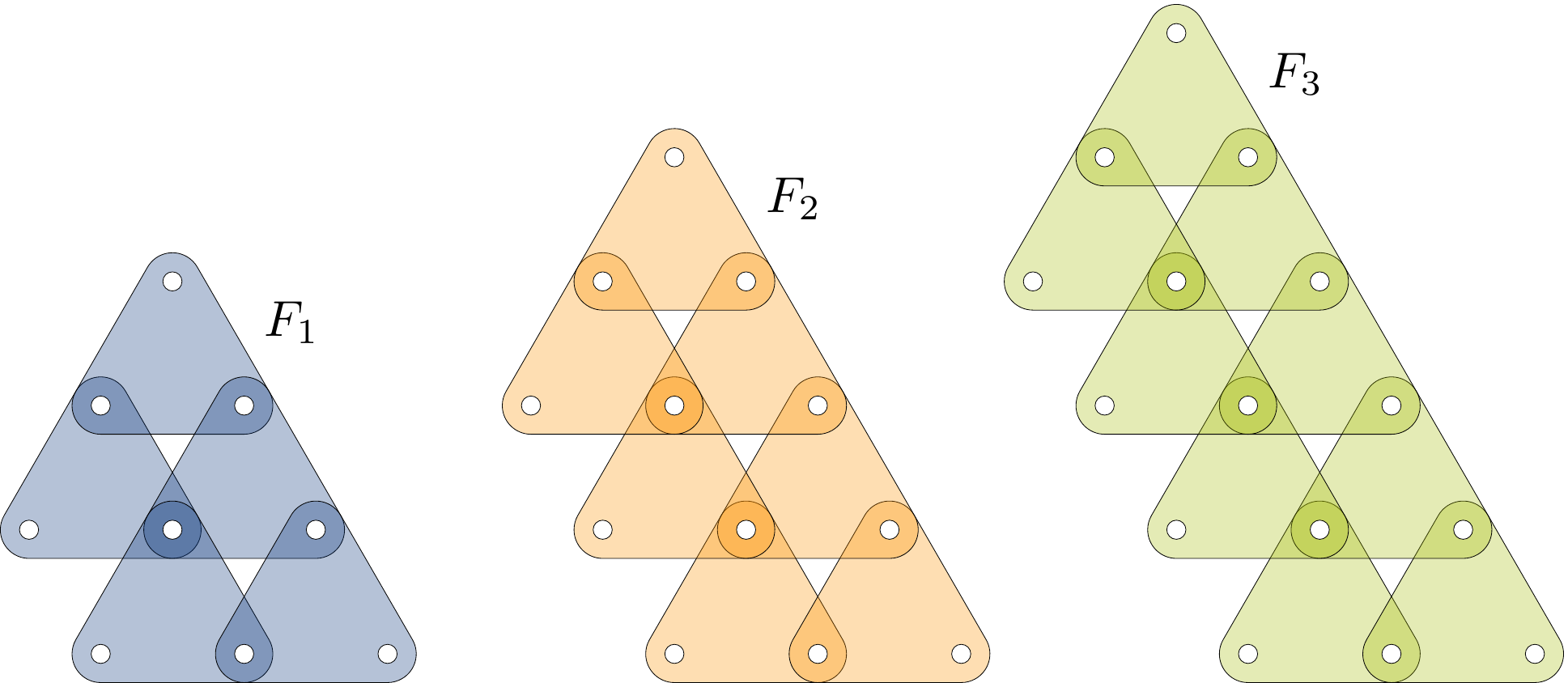}
        \caption{The graph family $F_i$ used in the proof of Lemma~\ref{lemma:excisive_and_functorial_not_representable}, illustrated by its first three members.}
        \label{fig:appendix_graphs_Fi}
    \end{figure}

    Now consider the family $F_i$ of graphs as in Figure~\ref{fig:appendix_graphs_Fi}. We claim that if $k > 3i$, we must have $F_i \in \Rset$. We proceed by induction -- so, for the base case, suppose that $k > 3$ and consider the graph $F_1$. It is easy to see that $\Clust(F_1)$ has a single part containing all vertices. Now, $\Phi_\Rset(F_1)$ contains at least two edges corresponding to the two copies of $D$ in $F_1$, but these overlap in only three vertices, and so there will not be an edge between them in $\Lk(\Phi_\Rset(F_1))$. So there must be some additional graph $G$ in $\Rset$ that is a subgraph of $F_1$, and overlaps both copies of $D$ in more than three vertices.

    If this $G$ is a proper subgraph of $F_1$, it is easy to see that it will not be sent to a single part by $\Clust$, which would be a contradiction to the fact that $\Pi_{\Rset,k}(G)$ would send this $G$ to a single part. So $G$ must be $F_1$ itself, and so $\Rset$ must contain $F_1$.

    For the inductive step, let us fix an $i$ and suppose that $k > 3i$. Then, by our induction hypothesis, $F_{i-1} \in \Rset$. Again, we have that $\Clust(F_i)$ has a single part containing all vertices. In $\Phi_\Rset(F_i)$ we will have two edges corresponding to the copy of $F_{i-1}$ in $F_i$ created by the first $i$ copies of $D$ and last $i$ copies of $D$ respectively, and these overlap in a copy of $F_{i-2}$, which has $3i < k$ vertices. So there is no edge connecting these two edges in $\Lk(\Phi_\Rset(F_i))$, and so there must be some additional graph $G$ in $\Rset$ that is a subgraph of $F_i$ and overlaps both copies of $D$ in more than $3i$ vertices. The same argument again gives that, if this subgraph were proper, it would be a copy of $F_{i-1}$ with a single edge hanging off of it, which would not be sent to a single part by $\Clust$, and so it must be $F_i$ itself.

    \paragraph*{Excluding the case $3 < k < \infty$:}

    So, suppose $3 < k < \infty$, and let $i$ be such that $3i < k \leq 3(i+1)$. From the preceding section, we know that $F_i \in \Rset$. Our strategy will be to find a graph $G_i$ that has two copies of $F_i$ in it which overlap in all but three vertices, but share no edge. Now, since $F_i$ has $3(i+2)$ vertices, and so the two copies of $F_i$ would be overlapping in $3(i+1) \geq k$ vertices, we would have that $\Pi_{\Rset,k}(G_i)$ would have a part containing all vertices, while $\Clust(G_i)$ would not, since the two copies of $F_i$ do not overlap in any full edge.

    \begin{figure}
        \centering
        \includegraphics[width=0.7\textwidth]{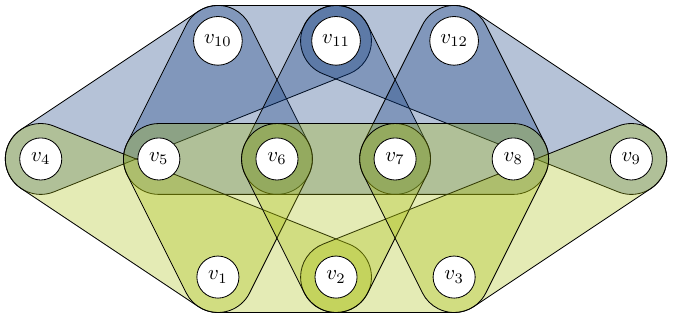}
        \caption{The graph $G_1$ used in the proof of Lemma~\ref{lemma:excisive_and_functorial_not_representable}. This graph contains two copies of $F_1$, corresponding to the green and the blue edges. Here, the vertices $v_4$ and $v_9$ are the leftmost column in the standard drawing of $F_1$ (Figure~\ref{fig:appendix_graphs_Fi}), and the vertices $v_5,v_6,v_7$ and $v_8$ are the rightmost column. The vertices $v_1,v_2,v_3$ are the middle column in the green copy of $F_1$, and the vertices $v_{10},v_{11},v_{12}$ are the middle column in the blue copy of $F_1$.}
        \label{fig:appendix_graph_G1}
    \end{figure}
    \clearpage

    The construction we use here gets too complicated to draw in any case beyond $i = 1$ (that case is illustrated in Figure~\ref{fig:appendix_graph_G1}), but the idea is the following: The graph $F_i$ has three ``columns'' of vertices, and we can take our two copies of $F_i$ to have the same left- and rightmost columns. For the middle column, we pick three of the vertices in each copy to be distinct, and for the rest, we permute their labels and then identify vertices with the same label. As long as we pick a permutation $\sigma$ such that $\abs{\sigma(j)-j}>1$ for all $i$, we will have that the two copies of $F_i$ do not share any edges.

    That such a permutation $\sigma$ exists for $i > 3$ is easy to see -- in that case, we can even pick such a permutation of the entire middle column first. In the case $i = 1$, there are no middle column vertices left once we remove the three distinct ones, and for the cases $i = 2, 3$, the reader is invited to look at Figure~\ref{fig:appendix_graphs_Fi} and verify that the construction can be performed then as well.

    \paragraph*{Excluding the case $k = \infty$:}

    Finally, we need to exclude the case $k = \infty$. In this case, we know that $\Rset$ must be precisely the set of all graphs that are sent to a partition with a part containing all vertices by $\Clust$, from the reasoning in the proof of Lemma~\ref{lemma_exc_and_funct_implies_repr}. But then, for example, the graph $F_1$ will be sent to a partition that has three parts -- one containing all vertices, and two corresponding to the copies of $D$ in $F_1$, while of course $\Clust(F_1)$ has only one part. So we have a contradiction.
\end{proof}

\section{Further examples}

In this section, we record additional interesting examples that do not fit into the main text.

\begin{example}\label{example:equal_parts_different_edgesets}
    There exists a graph $G$ such that $\Pi_2(G)$ has two parts, both of which contain all vertices of $G$, arising from two different connected components of $\Lambda_2(G)$.
\end{example}

\begin{figure}[p]
    \centering
    \includegraphics[width=0.8\textwidth]{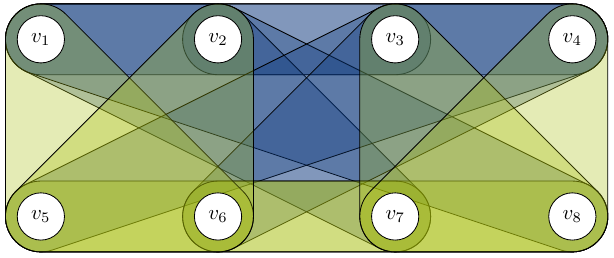}
    \caption{The graph $G$ used in the proof of Example~\ref{example:equal_parts_different_edgesets}.}
    \label{fig:equal_parts_different_edgesets}
\end{figure}

\begin{figure}[p]
    \centering
    \includegraphics[width=0.5\textwidth]{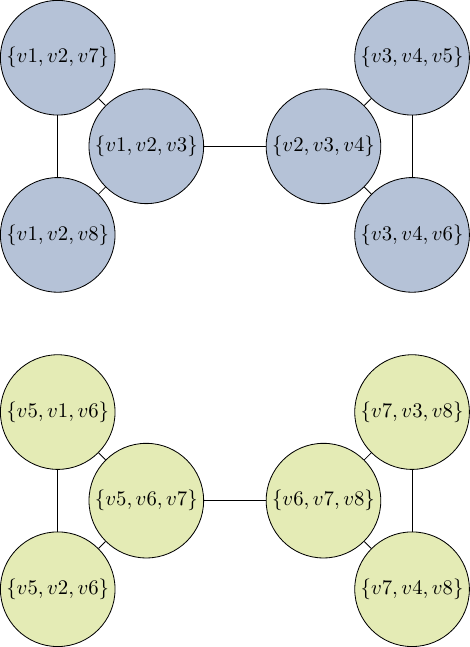}
    \caption{The line graph $\Lambda_2(G)$ of the graph $G$ in Figure~\ref{fig:equal_parts_different_edgesets}. The vertices are coloured according to the colour of the corresponding edge in Figure~\ref{fig:equal_parts_different_edgesets}.}
    \label{fig:equal_parts_different_edgesets_linegraph}
\end{figure}

\begin{proof}
    Consider the graph $G$ as in Figure~\ref{fig:equal_parts_different_edgesets}, with $2$-line graph as in Figure~\ref{fig:equal_parts_different_edgesets_linegraph}.
\end{proof}

\section{Proof of non-functoriality of $\Upsilon$}

If we wanted to turn $\Upsilon$ into a functor, we would need to define it on morphisms in $\Pset$. Now, thinking about how we use it to map from sets of edges to sets of vertices, what we would be doing is defining a map from vertices to vertices which was coherent with the mapping from edges to edges which we had been given. For this to be at all natural, we would want to have that $\Upsilon(f)_*(e) = f(e)$ for all edges $e$, so that we ``get the right thing when we go in the other direction''.

However, this is not possible, essentially because there are morphisms between line graphs that do not arise from morphisms between the original graphs. To see this, consider for example the graphs $G = (\{1,2,3\}, \{\{1,2\}, \{2,3\}\})$ and $H = (\{a,b\}, \{\{a\}, \{b\}\})$ -- $\Lambda_2(G)$ and $\Lambda_2(H)$ are both the graph with two vertices and no edges, and so they are isomorphic, but clearly $H$ and $G$ themselves are not isomorphic. This is essentially the idea of the proof of the following lemma.

\begin{lemma}\label{lemma:Upsilon_not_functorial}
    There is no way of defining $\Upsilon(f)$ on morphisms in $\Pset$ such that $\Upsilon$ becomes a functor, and $\Upsilon(f)_* = f$ for all morphisms $f$ in $\Pset$.
\end{lemma}
    
\begin{proof}
    Suppose, to the contrary, that $\Upsilon$ could be defined on morphisms, and let
    \[
    \begin{aligned}
    A &= \bigl(V_A,P_A\bigr),
    &&
    V_A=\bigl\{\{1,2\},\,\{2,3\}\bigr\},\quad
    P_A=\bigl\{\{\{1,2\}\},\,\{\{2,3\}\}\bigr\},\\
    B &= \bigl(V_B,P_B\bigr),
    &&
    V_B=\bigl\{\{a\},\,\{b\}\bigr\},\quad
    P_B=\bigl\{\{\{a\}\},\,\{\{b\}\}\bigr\}.
    \end{aligned}
    \]
    Define a morphism \(f\colon A\to B\) in \(\Pset\) by
    \[
    f(\{1,2\})=\{a\}, 
    \quad
    f(\{2,3\})=\{b\}.
    \]
    So when we apply \(\Upsilon\) to $A$ and $B$, we get
    \[
    \Upsilon(A)
    =\bigl(\{1,2,3\},\,\{\{1,2\},\{2,3\}\}\bigr),
    \qquad
    \Upsilon(B)
    =\bigl(\{a,b\},\,\{\{a\},\{b\}\}\bigr).
    \]
    
    Now, we see easily that the only morphisms from $\Upsilon(A)$ to $\Upsilon(B)$ are the two constant functions sending all vertices to $a$ or all vertices to $b$. So suppose without loss of generality that we have defined $\Upsilon(f)(x) = a$ for $x = 1,2,3$ -- but then we get $\Upsilon(f)_*(\{2,3\}) = \{a\}$, while $f(\{2,3\}) = \{b\}$. So we have a contradiction -- the opposite choice of $\Upsilon(f)$ of course gives the same contradiction, and so we are done.
\end{proof}

\end{document}